\documentclass[11pt]{article}
\usepackage{amsmath,amsthm,amsfonts,amssymb,amscd, amsxtra, mathrsfs}
\usepackage{url}
\usepackage{cancel} 
\usepackage{rotating}
\usepackage{tikz}
\usepackage{pgfplots}
\usepackage{pgfplotstable}
 \pgfplotsset{compat=newest}

\usepackage[margin=2.7 cm,nohead]{geometry}
\usepackage{color}
\usepackage{color}
\usepackage[pdftex]{hyperref}
\usepackage{graphicx,color,rotating}
\usepackage{amsmath,amsthm}
\usepackage{amssymb}

\theoremstyle{plain}
\newtheorem{theorem}{Theorem}[section]
\newtheorem{lemma}[theorem]{Lemma}
\newtheorem{corollary}[theorem]{Corollary}
\newtheorem{proposition}[theorem]{Proposition}

\theoremstyle{definition}
\newtheorem{definition}[theorem]{Definition}
\newtheorem{example}[theorem]{Example}

\theoremstyle{remark}
\newtheorem{remark}{Remark}

\newtheorem{algorithm}{Algorithm}

\def\diam{\operatorname{diam}}

\newcommand{\argmin}{\rm argmin}

\newcommand{\R}{\mathbb R}
\newcommand{\N}{\mathbb N}
\newcommand{\C}{{\cal C}}

\def \T {{\scriptscriptstyle\mathrm{T}}} 

\def \T {{\scriptscriptstyle\mathrm{T}}} 
\def\R{\mathbb R}
\begin{document}
\title{Frank--Wolfe algorithms for piecewise star-convex functions with a nonsmooth  difference-of-convex structure}

\author{
R. D\'iaz Mill\'an \thanks{School of Information Technology, Deakin University, Melbourne, Australia, e-mail:\url{r.diazmillan@deakin.edu.au}.}
\and
O.  P. Ferreira  \thanks{IME, Universidade Federal de Goi\'as,  Goi\^ania, Brazil, e-mail:\url{orizon@ufg.br}.}
\and
 J. Ugon  \thanks{School of Information Technology, Deakin University, Melbourne,   Australia,  e-mail:\url{j.ugon@deakin.edu.au}.}}
\maketitle
\begin{abstract}
We study Frank--Wolfe algorithms for minimizing a nonconvex function expressed as the difference of a differentiable convex function with a Lipschitz continuous gradient and a possibly nonsmooth convex function over a compact convex set.
We propose an adaptive Frank--Wolfe scheme that exploits this difference-of-convex structure and prove a sufficient descent property together with the convergence of the Frank--Wolfe gap to zero.
We also introduce the notion of piecewise star-convexity, and show that for this class of functions Clarke-stationary points are minimizers on the cells of the underlying partition and, for finite partitions, we establish $\mathcal{O}(1/k)$ complexity bounds in function value and in the Frank--Wolfe gap.
In particular, these bounds match the standard rates obtained in the convex case, showing that Frank--Wolfe methods can retain optimal complexity guarantees beyond the classical convex setting, even in the presence of nonsmooth terms.
\end{abstract}


\section{Introduction}
The DC optimization method involves minimizing a function that can be represented as the difference between two convex functions. We are interested in finding a solution to a constrained DC optimization problem, where the constraint set ${{\cal C}} \subset {\mathbb R}^n$ is a convex and compact set, and $f:=g-h$ where $g :\mathbb{R}^n \to \mathbb{R}$ are continuously differentiable convex functions, and $h:\mathbb{R}^n \to \mathbb{R}$ is   a convex function possibly non-differentiable.  To the best of our knowledge, DC optimization can be traced back to pioneering works such as \cite{TaoLe1997,Pham1986}, which introduced the first algorithms for this problem. Since then, the DC optimization has attracted the attention of the mathematical programming community (see for example   \cite{Boufi2020,deOliveira2020,Dong2021,LeThietall2018,Ackooi2019}),  not only for its own sake but also because it is an abstract model for several families of practical optimization problems. Although we are not concerned with practical issues at this time, we emphasize that practical applications emerge whenever the natural structure of the problem is modelled as a DC optimization problem, such as sparse generalized eigenvalue problems \cite{torres2011majorization}, sparse optimization problems \cite{Gotoh2018}, facility location and clustering problems \cite{Nam2017}.

The Frank-Wolfe algorithm has a long history, dating back to Frank and Wolfe's work in the 1950s to minimize convex quadratic functions over compact polyhedral sets,  see  \cite{FrankWolfe1956}.  This method was generalized about ten years later to minimize convex differentiable functions with Lipschitz continuous gradients and compact constraint convex sets, see   \cite{LevitinPolyak1966}. Since then, this method has attracted the attention of several researchers who work with continuous optimization, thus becoming also known as the conditional gradient method. One of the factors that explains the interest in this method is its simplicity and ease of implementation. In fact, each method iteration only requires access to a linear minimization oracle over a compact convex set. It is also worth noting that, due to the method's simplicity, it allows for low storage costs and ready exploration of separability and sparsity, making its application in large-scale problems quite  attractive. It is interesting to note that the popularity of this method has increased significantly in recent years as a result of the emergence of several applications in machine learning, see \cite{Jaggi2013, LacosteJaggi2015,Lan2013}.   For all of these reasons, several variants of this method have arisen throughout the years and new properties of it have been discovered, resulting in a large literature on it,  papers  dealing with this method include \cite{BeckTeboulle2004,Bouhamidietall2018, BoydRecht2017,FreundMazumder2017,  Ghadimi2019,HarchaouiNemirovski2015, Konnov2018, LanZhou2016,LussTeboulle2013}.   

In the present paper,   we formulate  a version of  {\it Frank--Wolfe algorithm or the conditional gradient method to solve  DC optimization problem with an adaptive step size}.    It is worth mentioning that the Frank-Wolfe algorithm has previously  been formulated in the context of DC programming, see  \cite{Khamaru2019}.  See also  \cite{Yurtsever2022cccp} which establishes  theoretical connections  between  DC algorithms  and  convex-concave procedures with the Frank-Wolfe algorithm.  In contrast to the previous study, which assume that the  curvature/Lipschitz-type  constant of $f = g-h$ is bounded from above, the analysis of the Frank-Wolfe method done here just assumes that the gradient of the first component $g$ of $f$ is Lipschitz continuous.  While designing the methods, even if we assume that the  gradient of  $g$  is Lipschitz continuous, we will not compute the step size using the Lipschitz constant of $\nabla g$. The step size will be computed adaptively, based on an idea introduced in \cite{Beck2015} (see also \cite{BeckTeboulle2009, PedregosaJaggi2020}) that approximates the Lipschitz constant. It has been shown in previous works  that Frank-Wolfe algorithm produces a stationary point with a convergence rate of only ${\cal O} (1/\sqrt{k})$ for non-convex objective functions,  see \cite{Lacoste2016} (see also \cite{Khamaru2019}). Here, we introduce the concept of   piecewise star-convex function notion, which generalizes the star-convex function concept proposed in \cite{NesterovPolyak2006}  as well as a few other related concepts, such as  \cite{Hinder2020, WangWibsono2023}.  The convergence analysis developed in this paper covers \emph{piecewise star--convex with the  nonsmooth  difference-of-convex structure} aforementined, we show that the proposed Frank--Wolfe scheme equipped with our backtracking stepsize attains the \emph{same} order of iteration complexity  as in the convex case,  both the function-value  and the Frank-Wolfe gap decay at the standard sublinear rate  \(\mathcal{O}(1/k)\).
This holds even though piecewise star--convex functions are, in general, nonconvex and may be nonsmooth across cell boundaries.
Consequently, our theory extends convex-type complexity guarantees for Frank--Wolfe methods to a broader nonconvex setting, including piecewise star--convex DC objectives with a smooth convex term $g$ and a convex, possibly nonsmooth term $h$.
These guarantees complement recent results for   star--convex functions in~\cite{Millan2025}, and indicate that projection-free Frank--Wolfe schemes can maintain the usual first-order complexity behavior beyond classical convexity, even in the presence of nonsmooth terms.

The remainder of the paper is organized as follows. Section~\ref{sec:Preliminares} introduces the notation and collects auxiliary results. 
Section~\ref{secc3} presents the optimization problem, the standing assumptions, and related notation. 
In Section~\ref{Sec:StarConvex} we define \emph{piecewise star--convexity}, establish its main properties, and provide illustrative examples. 
Section~\ref{Sec:FW} is devoted to the proposed Frank--Wolfe algorithm, including its formulation, well-posedness, and the basic properties of the generated sequence.  Section~\ref{Sec:StarConvexConvProp} analyzes the convergence properties of the method, and Section~\ref{Sec:StarConvexIntComplexity} derives the corresponding iteration--complexity bounds. 
Finally, Section~\ref{sec:conclusions} summarizes our conclusions.

\section{Preliminaries} \label{sec:Preliminares}
In this section, we recall   some notations, definitions and basic results used throughout  the paper. A function $\varphi:\mathbb{R}^{n}\to  \mathbb{R} $ is said to be {\it convex}  if $\varphi(\lambda x + (1-\lambda)y)\leq \lambda \varphi(x) + (1-\lambda) \varphi(y)$, for all $x,y\in {\mathbb{R}^{n}}$ and  $\lambda \in [0,1]$.  And  $\varphi$ is \emph{strictly convex}  when the  last inequality is strict for $x\neq y$, for a comprehensive study of convex functions see \cite{Lemarechal}.  We say that $f:\mathbb{R}^{n}\to\mathbb{R}$ is \emph{locally Lipschitz} if, for all $x\in \mathbb{R}^{n}$, there exist a constant ${\cal C} _{x}>0$ and a neighborhood $U_{x}$ of $x$ such that $|f(x)-f(y)|\leq {\cal C} _{x}\|x-y\|$, for all $y\in U_{x}.$ It is well known that, if $f:\mathbb{R}^{n}\to \mathbb{R}$ is convex, then $f$ is locally Lipschitz. And we also know that if $f:\mathbb{R}^{n}\to\mathbb{R}$ is continuously differentiable then $f$ is locally Lipschitz, see {\cite[p. 32]{clarke1983optimization}}. Let $f:\mathbb{R}^{n}\to\mathbb{R}$ be a locally Lipschitz function. The {\emph{Clarke's subdifferential}} of $f$  at  $x\in \mathbb{R}^{n}$ is  given by $\partial^{c} f(x)=\{v \in \mathbb{R}^{n}:~ f^{\circ}(x;d)\geq  v^{\T}d, ~
\forall d \in \mathbb{R}^{n} \},$
where $f^{\circ}(x;d)$ is the {\emph{generalized directional derivative}} of $f$ at $x$ in the direction $d$ given by
$$ f^{\circ}(x;d)=  
\limsup _{ 
\tiny{\begin{array}{c}
u\rightarrow x\\
t\downarrow 0
\end{array}}} \frac{ f(u+td)-f(u) }{t}.
$$ 
For an extensive study of locally Lipschitz functions and Clarke's subdifferential   see \cite[p. 27]{clarke1983optimization}.  If $f$ is convex, then $\partial^{c}f(x)$ coincides with the subdifferential $\partial f(x)$ in the sense of convex analysis, and $f^{\circ}(x;d)$ coincides with the usual directional derivative $f'(x;d)$; see {\cite[p. 36]{clarke1983optimization}}. We recall that if $f:\mathbb{R}^{n}\to  \mathbb{R} $ is differentiable, then $\partial^{c} f(x)=\{\nabla f(x)\}$ for any $x\in \mathbb{R}^{n}$; see {\cite[p. 33]{clarke1983optimization}}. 

\begin{theorem}[{\cite[p. 27]{clarke1983optimization}}] \label{th:cdd}
Let $f:\mathbb{R}^{n}\to\mathbb{R}$ be a locally Lipschitz function. Then,  $\partial^{c}f(x)$ is a nonempty, convex, compact subset of $\mathbb{R}^{n}$ and $\|v\|\leq {\cal C} _{x},$ for all $v\in \partial^{c}f(x)$, where ${\cal C} _{x}>0$ is the Lipschitz constant of $f$ around $x$. Moreover,  $f^{\circ}(x;d) = \max \{ v^{\T}d :~ v\in \partial^{c}f(x)\} $.
\end{theorem}
\begin{theorem}[{\cite[p. 38-39]{clarke1983optimization}}]\label{subdif_DC}
Let  $f:\mathbb{R}^n \to \mathbb{R} $  be given by $f=g-h$, where  $g,h:\mathbb{R}^n \to \mathbb{R} $ is locally Lipschitz   functions and  $g$ is   differentiable.   Then,  $f^{\circ}(x;d)=\nabla g(x)^{\T}d -h'(x;d)$,  for all  $x,d\in \mathbb{R}^{n}$ and $\partial^{c}f(x) =\{\nabla g(x)\}-\partial h(x)$.
\end{theorem}
The next result is a combination of Theorem~\ref{subdif_DC} with {\cite[Corollary on p. 52]{clarke1983optimization}}. 
\begin{theorem}\label{th:opc}
Let $f:\mathbb{R}^{n}\to\mathbb{R}$ be a locally Lipschitz function and ${{\cal C}} \subset {\mathbb R}^n$ is a closed   and convex set. If $x^*\in {{\cal C}}$ is a minimizer of $f$ in ${{\cal C}}$, then  there exists $v\in \partial^{c} f(x^*)$ such that $v^{\T}(x-{x^*})\geq 0$, for all  $x\in {{\cal C}}$. As a consequence, if $f=g-h$ with   $g,h:\mathbb{R}^n \to \mathbb{R} $ is locally Lipschitz   functions and  $g$ is   differentiable, then   there exists  $u\in \partial^{c}h({x^*})$  such that $(\nabla g({x^*})-u)^{\T}(x-{x^*})\geq 0$, all $x\in {{\cal C}}$.
\end{theorem}
Hence, according to Theorem~\ref{th:opc}, every point satisfying the following inequality $v^{\T}(x-{x^*})\geq 0$,  for some  $v\in \partial^{c}f({x^*})$ and  for all $x\in {\cal C}$, is called to be a  {\it stationary point} of $\min_{x\in {{\cal C}}} f(x)$. A continuously differentiable function $g: \mathbb{R}^n \to \mathbb{R}$ has  gradient  $\nabla g$  is {\it  $L$-Lipschitz continuous} on ${{\cal C}} \subset {\mathbb R}^n$ if    there  exists a Lipschitz  constant $L>0$ such that  $\| \nabla g(x)- \nabla g(y) \| \leq L\|x-y \|$ for all ~ $x,y \in {{\cal C}}$. Thus, by using the  fundamental theorem of calculus, we obtain the following result whose proof can be found in \cite[Proposition A.24]{Bertsekas1999}, see also\cite[Lemma~2.4.2]{DennisSchnabel1996}.
\begin{proposition} \label{le:DescentLemma}
	Let $g: \mathbb{R}^n \to \mathbb{R}$ be a differentiable  with gradient  $L$-Lipschitz continuous on ${{\cal C}} \subset {\mathbb R}^n$, $x \in {{\cal C}}$, $v\in \mathbb{R}^n$  and $\lambda \in [0,1]$. If  $x+\lambda v \in {{\cal C}}$, then 
$
g(x+\lambda v) \leq g(x) + \nabla g(x)^{\T}v \lambda + \frac{L}{2}\|v\|^2\lambda^{2}.
$
\end{proposition}
\begin{proposition}\label{le:cc}
The function $h:\mathbb{R}^{n}\to  \mathbb{R} $ is convex if and only if   $h(y)\geq h(x) + \langle u, y-x \rangle$, for all $x,y\in \mathbb{R} ^{n} $ and  all $u\in \partial  h(x)$.
\end{proposition}
\begin{proposition}[{\cite[Proposition 6.2.1]{Lemarechal}}] \label{pr:csubdiff}
Let $h:\mathbb{R}^{n}\rightarrow \mathbb{R}$ be convex.   Let $\{x^k\}_{k\in\mathbb{N}}$ and  $(u^k)_{k\in{\mathbb N}}$	be   sequences such that  $u^k\in \partial h(x^k)$, for all $k\in \mathbb{N}$. If $\lim_{k\to +\infty}x^k={\bar x}$ and $\lim_{k\to +\infty}u^k={\bar u}$, then ${\bar u}\in \partial h({\bar x})$.
\end{proposition}
\begin{proposition}[{\cite[Proposition 6.2.2]{Lemarechal}}] \label{pr:bsubdiff}
Let $h:\mathbb{R}^{n}\rightarrow \mathbb{R}$ be convex. The mapping $\partial h$ is locally bounded, i.e. the image  $\partial h(B)$ of a bounded set $B\subset \mathbb{R}^{n}$ is a bounded set in $\mathbb{R}^{n}$.
\end{proposition}
In the following we recall an useful results   for our study on iteration complexity bounds for Frank--Wolfe algorithm, its proof can be found in  \cite[Lemma~13.13, Ch. 13, p. 387]{Beck2017}.
\begin{lemma}\label{lemma taxa}
Let $(a_k)_{k\in{\mathbb N}}$	and  $(b_k)_{k\in{\mathbb N}}$ be  nonnegative sequences of real numbers satisfying 
$$
a_{k+1}\leq a_{k}-b_{k}\beta_{k}+\frac{A}{2}\beta_{k}^{2}, \qquad k=0, 1, 2, \ldots, 
$$
where $\beta_{k}=2/(k+2)$ and $A$ is a positive number. Suppose that $a_{k}\leq b_{k}$, for all $k$. Then 
\begin{item}
\item[(i)] 	$\displaystyle a_k \leq \frac{2A}{k}$, for all $k=1, 2, \ldots.$
\item[(ii)] $\displaystyle \min_{\ell\in\{\lfloor\frac{k}{2}\rfloor+2,\ldots,k\}} b_\ell \leq \frac{8A}{k-2}$, for all $k=3, 4, \ldots,$ where, $\lfloor k/2 \rfloor= \max \left\lbrace n\in \mathbb{N} :~ n\leq k/2\right\rbrace.$ 
\end{item}
\end{lemma}

We conclude this section by recalling the definitions of some norms that will be used later. 
For $x = (x_1,\dots,x_n)^\top \in \mathbb{R}^n$ and $p \in [1,+\infty)$, the \emph{$p$--norm} is 
$\|x\|_p := \left( \sum_{i=1}^n |x_i|^p \right)^{\!1/p}$. In particular, the \emph{$\ell_1$--norm} is 
$\|x\|_1 := \sum_{i=1}^n |x_i|$, and the \emph{$\ell_\infty$--norm} is 
$\|x\|_\infty := \max_{1 \le i \le n} |x_i|$.

\section{The DC optimization  problem} \label{secc3}
We are interested in solving the following constrained  DC optimization problem
\begin{equation}\label{pr:main}
\begin{array}{c}
\min_{x\in {{\cal C}}} f(x):=g(x)-h(x), 
\end{array}
\end{equation}
where  ${{\cal C}} \subset {\mathbb R}^n$ is a compact  and convex set,  $g :\mathbb{R}^n \to \mathbb{R} $ is a  continuously differentiable  convex  function  and $h:\mathbb{R}^n \to \mathbb{R} $  is  convex function possibly non-differentiable, with domain on $\mathbb{R}^n$.   {\it  Throughout the paper we assume that  the gradient  $\nabla g$  is {\it  $L$-Lipschitz continuous} on ${{\cal C}} \subset {\mathbb R}^n$}, i.e.,    there  exists a Lipschitz  constant $L>0$ such that 
\begin{itemize}
\item[{\bf (A)}] $\| \nabla g(x)- \nabla g(y) \| \leq L\|x-y \|$ for all ~ $x,y \in {{\cal C}}$.
\end{itemize}
Since we are assuming that   ${\cal C} \subset {\mathbb R}^n$ is a compact set,   its {\it diameter}  is a finite number defined by
\begin{equation} \label{eq:diam}
 \diam({\cal C}):= \max\left\{ \|x-y\|:~x, y\in {\cal C}\right\}. 
\end{equation} 
Since ${\cal C}$ is compact and $f$ is continuous on ${\cal C}$, the problem~\eqref{pr:main} admits a finite global optimal value
\begin{equation} \label{eq:globalmin}
f_{\cal C}^*:=\min_{x\in{\cal C}} f(x) \;>\;-\infty,
\end{equation} 
and the set of minimizers ${\cal C}^*:=\arg\min_{x\in{\cal C}} f(x)$ is nonempty.
By Theorem~\ref{th:opc}, a point $\bar x\in{\cal C}$ is said to be \emph{Clarke stationary} for problem~\eqref{pr:main} if there exists
$\bar u\in \partial^{c}h(\bar x)$ such that
\begin{equation}\label{eq:oc}
\bigl(\nabla g(\bar x)-\bar u\bigr)^{\top}(x-\bar x)\;\ge\;0,
\qquad \forall\,x\in{\cal C}.
\end{equation}
In general,~\eqref{eq:oc} is a necessary condition for optimality but may fail to be sufficient.
In particular, every minimizer $x^*\in{\cal C}^*$ satisfies~\eqref{eq:oc}.

Let us introduce the linear minimization oracle and the associated gap function that will be used to define the Frank--Wolfe step for the DC model~\eqref{pr:main}.
Since $f=g-h$ may be nonsmooth, the linearization is built from $\nabla g(x)$ and a subgradient of $h$.
Given $x\in{\cal C}$ and $u\in\partial h(x)$, define
\begin{equation}\label{eq:LO}
   p(x,u)\in \arg \min_{v\in{\cal C}} \bigl(\nabla g(x)-u\bigr)^{\top}(v-x),
   \qquad
   \omega(x,u):=\bigl(\nabla g(x)-u\bigr)^{\top}\bigl(p(x,u)-x\bigr).
\end{equation}
When no ambiguity arises, we write $p(x)$ and $\omega(x)$ for $p(x,u)$ and $\omega(x,u)$, where $u$ is the algorithmic choice at $x$.

\begin{proposition}\label{prop:omega}
Let ${\cal C}\subset\R^n$ be compact and convex. For every $x\in{\cal C}$ and every $u\in\partial^{c}h(x)$,
let $p(x,u)$ and $\omega(x,u)$ be defined by~\eqref{eq:LO}. Then
$
\omega(x,u)\le 0.
$
Moreover, if $\omega(x,u)=0$, then $x$ is a stationary point for~\eqref{pr:main}.
Conversely, if $x\in{\cal C}$ is stationary, then there exists $u\in\partial^{c}h(x)$ such that
$\omega(x,u)=0$.
\end{proposition}

\begin{proof}
Fix $x\in{\cal C}$ and $u\in\partial^{c}h(x)$. Since the point $v=x$ is feasible in~\eqref{eq:LO}, we have
\[
\omega(x,u)
=\min_{v\in{\cal C}} \bigl(\nabla g(x)-u\bigr)^{\top}(v-x)
\le \bigl(\nabla g(x)-u\bigr)^{\top}(x-x)=0,
\]
which proves $\omega(x,u)\le 0$.

Furthermore, by~\eqref{eq:oc}, there exists $\bar u\in \partial^{c}h(\bar x)$ such that $w(x,u)\geq 0$ if and only if $x$ is Clarke stationary, which completes the proof. 
\end{proof}

We conclude this section with a standard example of a nonsmooth DC  function, see a particular instance in \cite[Example 2.4.]{ARAGON2019}. It satisfies our standing assumptions. In particular, it admits a decomposition $f=g-h$ in which $g$ is convex, differentiable, and has a Lipschitz continuous gradient, while $h$ is convex and possibly nonsmooth.
\begin{example}\label{ex:l2-l1-datafit}
Let $A\in\mathbb{R}^{n\times n}$ have full column rank, so that $Q:=A^\top A$ is positive definite.
Let $b\in\mathbb{R}^n$ and $\beta>0$. Define $f:{\mathbb R}^n \to\mathbb{R}$ by
\[
  f(x)=\tfrac12\,\|Ax-b\|_2^{2}-\beta\|x\|_1.
\]
Then $f$ is a nonsmooth DC function with the decomposition $f=g-h$, where $g(x):=\tfrac12\,\|Ax-b\|_2^{2}$ and $h(x):=\beta\|x\|_1$. 
Moreover, $g$ is convex and differentiable with a Lipschitz continuous gradient, while $h$ is convex and nonsmooth.
\end{example}

Beyond its DC structure, this example illustrates a geometric feature that will be central in the sequel. Although $f$ is nonconvex, it becomes well behaved when restricted to suitable convex regions. Indeed, the sign pattern of $x$ induces a natural partition of $\mathbb{R}^n$ into orthants, and on each orthant the function reduces to a strictly convex quadratic. Consequently, any Clarke-stationary point is a minimizer of $f$ on its corresponding cell, even if it is not a global minimizer on $\mathbb{R}^n$. This motivates the class of \emph{piecewise star-convex models} introduced in the next section, where star-type inequalities hold on each cell of a convex partition and can be used to obtain sharper convergence and complexity guarantees for Frank--Wolfe methods.

\section{Piecewise star-convexity with a nonsmooth DC structure} \label{Sec:StarConvex}
In this section, we introduce the notion of \emph{piecewise star-convexity}, a relaxation of the classical concept of star convexity from \cite{NesterovPolyak2006}. Although a difference-of-convex function is generally nonconvex, it often exhibits this weaker geometric property when restricted to suitable regions of the domain. As we show later, identifying and exploiting piecewise star-convexity yields sublinear convergence guarantees for Frank--Wolfe methods. Moreover, when the domain admits a finite partition, we obtain an ${\cal O}(1/k)$ rate on each region that matches the standard rate of the convex case, despite moving beyond the classical convex setting and allowing nonsmooth terms.

We begin by recalling the definition of star convexity from \cite{NesterovPolyak2006} and then introduce its piecewise counterpart, which will play a central role in our analysis.
\begin{definition}\label{def:star-conv}
Let ${{\cal C}} \subset {\mathbb R}^n$ be a convex set. A function
$f:{\mathbb R}^n \to\mathbb{R}$ is \emph{star--convex on ${{\cal C}}$} if its set
of global minimizers $X^*:=\arg\min_{{\cal C}} f$ is nonempty and, for any
$x^*\in X^*$,
\begin{equation}\label{eq:star-conv}
  f(\lambda x^*+(1-\lambda)x)
  \leq  \lambda f(x^*)+(1-\lambda)f(x),
  \qquad \forall\,x\in {{\cal C}},\ \forall\,\lambda\in [0,1].
\end{equation}
\end{definition}

 Every convex function with a nonempty set of global minimizers is star--convex,  but the converse does not hold.  We present two nonconvex star--convex functions from \cite{NesterovPolyak2006}  (see also \cite{Millan2025} for additional examples) to highlight that they are  DC functions, each expressible as the difference of two convex functions.

\begin{example}
The star--convex function $\phi(t)=|t|(1-e^{-|t|})$ is not convex. In addition, $\phi$ is a difference of two convex functions.  Indeed, considering the following two convex functions   $\varphi(t)=|t|(1-e^{-|t|})+ e^{|t|}$ and $\psi(t)= e^{|t|}$,  we have $\phi=\varphi-\psi$.  Similarly, consider the star--convex function  
$f(s,t)=s^2t^2+s^2+t^2$. The function $f$ is not convex,  but is a difference of two convex functions.  Indeed, letting the following convex functions  $g(s,t)=(s^2+t^2)^2+s^2+t^2$ and $h(s,t)=s^2t^2+s^4+t^4$,  we have $f=g-h$.
\end{example}

The condition in \eqref{eq:star-conv} is stated with respect to a \emph{fixed} minimizer $x^\ast$ and must hold for all $x\in{\cal C}$.
A milder relaxation is to allow the reference minimizer to depend on $x$, a behavior that often arises in DC models and is captured by piecewise star-convexity.
When $f$ has several ``basins'' of favorable geometry, even this can be too restrictive globally, so we partition the domain into regions where a star-type inequality holds relative to a local minimizer:

\begin{definition}\label{def:wsc-piecewise}
Let ${\cal C}\subset\mathbb{R}^n$ be convex and let ${\cal P}=\{\,{\cal C}_i\,:\,i\in I\}$ be a partition of ${\cal C}$ into nonempty convex subsets  and $\bigcup_{i\in I}{\cal C}_i={\cal C}$. We say that $f:\mathbb{R}^n\to\mathbb{R}$ is \emph{piecewise-star-convex on
${\cal C}$ with respect to ${\cal P}$} if, for every $i\in I$,
\begin{itemize}
\item[(i)] the set of global minimizers of $f$ on ${\cal C}_i$, denoted $X_i^*$, is nonempty; and
\item[(ii)] for each $x\in{\cal C}_i$ there exists $x_x^*\in X_i^*$ such that
\begin{equation}\label{eq:wsc-piecewise}
  f\bigl(\lambda x_x^*+(1-\lambda)x\bigr)
  \leq  \lambda f(x_x^*)+(1-\lambda)f(x),
  \qquad \forall\,\lambda\in[0,1].
\end{equation}
\end{itemize}
\end{definition}

\begin{remark}
When the partition is trivial, that is, ${\cal P}=\{{\cal C}\}$ and $X_1^*=X^*$, Definition~\ref{def:wsc-piecewise} reduces to Definition~\ref{def:star-conv}. In many applications, piecewise star-convexity arises because $f$ has a favorable geometry on each region ${\cal C}_i$ (for instance, it may be convex or star-convex when restricted to ${\cal C}_i$) but this geometry does not persist across different regions. Accordingly, the sets $X_i^*=\arg\min_{{\cal C}_i} f$ consist of \emph{global} minimizers of the cellwise problems, yet such points need not be global minimizers of $f$ on ${\cal C}$ and may fail to be even local minimizers of $f$ on ${\cal C}$. The key point is that the star-type inequalities in Definition~\ref{def:wsc-piecewise} are enforced \emph{within each cell} relative to $X_i^*$, while no compatibility is required when one crosses cell boundaries.
\end{remark}

We emphasize that, in Definition~\ref{def:wsc-piecewise}, the collection of subsets in item~(ii) need not be finite. Piecewise star-convexity may still hold when ${\cal C}$ is covered by \emph{infinitely many} convex regions. The following example exhibits an infinite partition and an objective function with a nonsmooth difference-of-convex structure.

\begin{example} \label{rk:converse-fails}
Consider the set  ${\cal C} = [0,1]$ and the infinite set
\(
   S := \{0\} \cup \{1/k : k \in \mathbb{N}\} \subset [0,1],
\)
which is compact, infinite, and has an accumulation point at $0$.
Define  \(f: [0,1] \to \mathbb{R}\) by 
\[
   f(x)=\min_{c\in S} |x-c|^2,\qquad x\in[0,1], 
\]
see the Figure~\ref{fig:weakstar-envelope}.  The set of global minimizers is 
$X^* = \operatorname*{argmin}_{x \in [0,1]} f(x) = S$ with $\min f = 0$.
For each $x \in [0,1]$, choose $x_x^* \in S$ attaining the minimum in the definition of $f$, 
and let
\[
   V_{x_x^*}:= \{\,y \in [0,1] : |y - x_x^*| \le |y - c| \ \forall c \in S\,\}.
\]
These sets form a countable partition 
\([0,1] = \bigcup_{k=0}^{\infty} V_k\), with
\[
   V_1 = \Bigl[ \tfrac34,\, 1 \Bigr], \quad
   V_k= \Bigl[ \tfrac{1/k + 1/(k+1)}{2},\ \tfrac{1/k + 1/(k-1)}{2} \Bigr], \ k \ge 2, \quad
   V_0= \{0\}.
\]
By construction, $x \in V_{x_x^*}$ and the whole segment between $x_x^*$ and  $x$ lies in $V_{x_x^*}$.  
Moreover, $f(y) = |y - x_x^*|^2$ for all $y$  belonging the segment between $x_x^*$ and  $x$.  
Hence, for any $\lambda \in [0,1]$,
\[
   f\big(\lambda x_x^* + (1-\lambda)x\big)
   = \big|\lambda x_x^* + (1-\lambda) x - x_x^*\big|^2
   \le \lambda f(x_x^*) + (1-\lambda) f(x).
\]
Since $f(x_x^*) = 0$, the piecewise-star-convexity inequality holds.  
However, covering ${\cal C}$ by convex sets satisfying (i)–(ii) requires all sets \(V_k\) above, 
so any such cover must be infinite.

\begin{figure}[ht]
\centering
\begin{tikzpicture}[x=11cm, y=45cm]

\draw[line width=0.8pt, ->, >=latex] (-0.05,0) -- (1.05,0) node[below] {$x$};
\draw[line width=0.8pt, ->, >=latex] (0,-0.01) -- (0,0.085) node[left] {$f(x)$};

\draw[very thick, domain=0.1339285714:0.1547619048, samples=120, smooth] plot (\x, {(\x-0.1428571429)^2}); 
\draw[very thick, domain=0.1547619048:0.1833333333, samples=120, smooth] plot (\x, {(\x-0.1666666667)^2}); 
\draw[very thick, domain=0.1833333333:0.225, samples=120, smooth]       plot (\x, {(\x-0.2)^2});          
\draw[very thick, domain=0.225:0.2916666667, samples=120, smooth]       plot (\x, {(\x-0.25)^2});         
\draw[very thick, domain=0.2916666667:0.4166666667, samples=120, smooth] plot (\x, {(\x-0.3333333333)^2});
\draw[very thick, domain=0.4166666667:0.75, samples=120, smooth]        plot (\x, {(\x-0.5)^2});          
\draw[very thick, domain=0.75:1, samples=120, smooth]                   plot (\x, {(\x-1)^2});            

\draw[dashed, gray] (0.1339285714,0) -- (0.1339285714,0.08);
\draw[dashed, gray] (0.1547619048,0) -- (0.1547619048,0.08);
\draw[dashed, gray] (0.1833333333,0) -- (0.1833333333,0.08);
\draw[dashed, gray] (0.225,0) -- (0.225,0.08);
\draw[dashed, gray] (0.2916666667,0) -- (0.2916666667,0.08);
\draw[dashed, gray] (0.4166666667,0) -- (0.4166666667,0.08);
\draw[dashed, gray] (0.75,0) -- (0.75,0.08);

\foreach \xx in {0,1,0.5,0.3333333333,0.25,0.2,0.1666666667,0.1428571429,0.125} {
  \fill (\xx,0) circle[radius=1.2pt];
}

\foreach \xx/\lab in {1/{1},0.5/{\tfrac{1}{2}},0.3333333333/{\tfrac{1}{3}},0.25/{\tfrac{1}{4}}} {
  \draw[line width=0.4pt] (\xx,0) -- (\xx,-0.0025);
  \node[below=2pt, font=\scriptsize] at (\xx,0) {\(\lab\)};
}

\end{tikzpicture}
\caption{%
Plot of the function \( f(x) = \min_{c \in S} (x-c)^2 \) for  
\( S = \{0, \tfrac18, \tfrac17, \tfrac16, \tfrac15, \tfrac14, \tfrac13, \tfrac12, 1\} \).  
The black segments indicate the active arc of the lower envelope, the dashed lines mark the boundaries of \( V(x_x^*) \),  
and the solid points on the \(x\)-axis represent the centers \( c \in S \).}
\label{fig:weakstar-envelope}
\end{figure}
\end{example}

We next present a prototypical example of a piecewise star-convex function with a \emph{nonsmooth DC structure}. This model arises naturally in sparse estimation, robust quadratic fitting, and regularized control. The $\ell_1$ term is convex but nonsmooth, and it induces a sign-dependent linear perturbation of a positive-definite quadratic. By partitioning $\mathbb{R}^n$ into orthants with fixed coordinate signs, the objective reduces on each region to a strictly convex quadratic with a unique minimizer. These cell minimizers provide the reference points in Definition~\ref{def:wsc-piecewise}, which establishes piecewise star-convexity even though, globally, the problem remains nonconvex. For use in the examples below and to simplify notation, define  
\begin{equation} \label{eq:ortant}
\{-1,1\}^n= \{ s:=(s_1,\ldots,s_n)\in \mathbb{R}^n:~  s_i\in \{-1,1\},~ \forall i\in \{1,\ldots,n\}\}.
\end{equation}

\begin{example} \label{ex:l2-l1}
Let $Q\in\R^{n\times n}$ be positive definite and let $\beta>0$. Consider $f:{\mathbb R}^n \to\mathbb{R}$ given by 
\begin{equation}\label{eq:deff}
    f(x) = \tfrac12\,x^\top Q x - \beta\|x\|_1, 
    \qquad x \in \R^n .
\end{equation}
For each \(s\in \{-1,1\}^n\), define the closed orthant
\[
   \Omega_s = \bigl\{\, x\in\R^n :~ s_i x_i \ge 0 \text{ for } i=1,\ldots,n \bigr\}, \qquad \qquad
   \R^n = \bigcup_{s\in\{-1,1\}^n}\Omega_s.
\]
On each $\Omega_s$ the signs of the coordinates are fixed and
$\|x\|_1=\sum_i |x_i|=\sum_i s_i x_i = s^\top x$ for $x\in\Omega_s$. Hence the
restriction of $f$ to $\Omega_s$ is the strictly convex quadratic
\begin{equation} \label{eq:oecell}
   f_s(x) = \tfrac12\,x^\top Q x - \beta\, s^\top x, 
   \qquad x\in \Omega_s,
\end{equation}
with constant Hessian $Q$. Since $Q$ is positive definite, $f_s$ is \emph{strongly convex} on $\Omega_s$ with modulus $\lambda_{\min}(Q)$; hence the minimizer $x_s^*$ is
\emph{unique}. Considering that $f$ coincides with the strictly convex quadratic $f_s$ on each
orthant $\Omega_s$ and $x_s^*$ is the (unique) minimizer of $f$ on $\Omega_s$,
the convexity inequality
\[
   f\bigl(\lambda x_s^* + (1-\lambda)x\bigr) 
   \leq  \lambda f(x_s^*) + (1-\lambda) f(x),
   \qquad \forall\, x\in\Omega_s,\ \forall\,\lambda\in[0,1],
\]
holds. Therefore $f$ is \emph{piecewise–star–convex} on $\R^n$ with respect to the full partition $\{\Omega_s\}_{s\in\{-1,1\}^n}$. A particularly interesting case is when $Q=\alpha I_n$ with $\alpha>0$. In this case, it follows from \eqref{eq:oecell} that on each orthant $\Omega_s$ we have
\(
f_s(x)=\tfrac{\alpha}{2}\,\|x\|_2^2-\beta\,s^\top x,
\)
whose unique minimizer on $\Omega_s$ solves $\alpha x-\beta s=0$, hence
\(
x_s^*=(\beta/\alpha)\,s \in \Omega_s .
\)
Therefore all $2^n$ points $(\beta/\alpha)\,s$, with $s\in\{-1,1\}^n$, are global minimizers of the function $f$, and the minimum value is
\(
\min_{x\in\R^n} f(x)= -\,{n\beta^{2}}/{(2\alpha)},
\)
which is independent of $s$.
\end{example}

The next example shows how the minimum of convex quadratics induces a polyhedral  partition of the space, a construction closely related to Voronoi diagrams and with  applications in clustering, facility location, geometry of distance and constrained optimization, see \cite{Libertietall2014} and  \cite{Aurenhammer1991} for a comprehensive survey on this subject. This construction is also connected to Example~\ref{rk:converse-fails}.

\begin{example} \label{ex:general}
Let $Q\in\mathbb{R}^{n\times n}$ be positive definite, and let $b_1,\dots,b_m\in\mathbb{R}^n$ and $c_1,\dots,c_m\in\mathbb{R}$.  Consider the function $f:{\mathbb R}^n \to\mathbb{R}$ defined by 
\begin{equation} \label{eq:pfq}
   f(x):=\min_{1\le i\le m}\varphi_i(x), \qquad  \qquad \varphi_i(x):=\tfrac12\,x^\top Q x + b_i^\top x + c_i,\qquad i=1,\dots,m.
\end{equation} 
Consider the partition ${\cal P}:=\{{\cal C}_i:~ i=1,\ldots,m\}$ given by
\begin{equation*}
  {\cal C}_i=\bigl\{x\in\mathbb{R}^n:\ \varphi_i(x)\le \varphi_j(x)\ \ \forall j=1,\dots,m\bigr\}
  =\bigcap_{j=1}^m \bigl\{x\in\mathbb{R}^n:\ (b_i-b_j)^\top x \le c_j-c_i \bigr\},
\end{equation*}
since $\varphi_i(x)-\varphi_j(x)=(b_i-b_j)^\top x+(c_i-c_j)$. Hence each ${\cal C}_i$ is a closed convex polyhedron. By construction, we have 
\[
  f(x)=\varphi_i(x)\qquad \forall\,x\in{\cal C}_i.
\]
Because $Q$ is positive definite, each $\varphi_i$ is strongly convex; therefore $f$ coincides with a strongly convex quadratic on each cell ${\cal C}_i$. In particular, $f$ is convex on each ${\cal C}_i$, hence \emph{star--convex} there, and thus \emph{piecewise star--convex} with respect to the polyhedral partition $\{{\cal C}_i\}_{i=1}^m$.

Next, observe that $f$ admits a nonsmooth DC representation $f=g-h$ with a smooth convex component $g$ and a convex (typically nonsmooth) component $h$. Indeed,
\[
\min_{1\le i\le m}\Bigl\{\tfrac12\,x^\top Q x + b_i^\top x + c_i\Bigr\}
= \tfrac12\,x^\top Q x + \min_{1\le i\le m}\bigl\{b_i^\top x + c_i\bigr\}
= \tfrac12\,x^\top Q x - \max_{1\le i\le m}\bigl\{-b_i^\top x - c_i\bigr\}.
\]
Therefore, we have $f(x)=g(x)-h(x)$ with the components given by $g(x):=\tfrac12\,x^\top Q x$ and $h(x):=\max_{1\le i\le m}\bigl(-b_i^\top x - c_i\bigr)$.
 Clearly $g$ is $C^\infty$ and convex, while $h$ is the pointwise maximum of affine functions, hence convex, polyhedral, and Lipschitz. 
 
  A particularly relevant special case of \eqref{eq:pfq} is obtained by taking pairwise distinct points $a_1,\dots,a_m \in \mathbb{R}^n$ and setting  \(\varphi_i(x)=\tfrac12(x-a_i)^\top Q (x-a_i)\). In this case, each cell $\mathcal C_i$ contains $a_i$, and
\(
  \min_{x\in\mathcal C_i} f(x)=\min_{x\in\mathcal C_i}\varphi_i(x)=\varphi_i(a_i)=0, 
\)
with the minimizer in \({\mathcal C_i}\) given by    \(x_i^\ast=a_i\).  Therefore,  all cells share the same minimum value $0$.
\end{example}

Next we show that a basic robustness property of piecewise–star–convexity is that the
\emph{star–type inequality} is preserved by adding a convex term. 

\begin{proposition} \label{prop:pwsc-convex-closed}
Let $f:\R^n\to\R$ be piecewise-star-convex on $\C$ with respect to the partition ${\cal P}=\{\C_i\}_{i\in I}$, and let $g:\R^n\to\R$ be convex on $\C$.  Then $f+g$ is piecewise-star-convex on $\C$ with respect to ${\cal P}$.
\end{proposition}
\begin{proof}
Fix $i\in I$ and $x\in\C_i$. By Definition~\ref{def:wsc-piecewise}, there exists 
$x_x^*\in X_i^*:=\arg\min_{\C_i} f$ such that
\[
  f\bigl(\lambda x_x^*+(1-\lambda)x\bigr)
  \le \lambda f(x_x^*) + (1-\lambda) f(x),\quad \forall\,\lambda\in[0,1].
\]
Since $g$ is convex on $\C_i$, we have $g\bigl(\lambda x_x^*+(1-\lambda)x\bigr) \leq \lambda g(x_x^*) + (1-\lambda) g(x)$, for all $\lambda\in[0,1]$.  Summing these  inequalities yields the star–type inequality for $f+g$ on $\C_i$ with the same 
base point $x_x^*$. As the argument holds for every $i\in I$ and $x\in\C_i$, the claim follows.
\end{proof}
\begin{remark}
The Example~\ref{ex:l2-l1-datafit}  follows immediately from Example~\ref{ex:l2-l1} and Proposition~\ref{prop:pwsc-convex-closed}.
\end{remark}

\begin{proposition} \label{prop:affine-AT}
Let $f:\mathbb{R}^n\to\mathbb{R}$ be piecewise–star–convex on ${\cal C}$ with respect to the partition ${\cal P}=\{\mathcal{C}_i\}_{i\in I}$, where each $\mathcal{C}_i$ is nonempty convex and $\bigcup_{i\in I}\mathcal{C}_i=\mathcal{C}$. Let $A\in\mathbb{R}^{n\times n}$ be invertible and $b\in\mathbb{R}^n$. Define the affine map $\Psi:\mathbb{R}^n \to \mathbb{R}^n$ by 
\[
\Psi(z):=Az+b,\qquad z\in\mathbb{R}^n,
\]
the transformed sets  $\tilde{\mathcal{C}}:=A^{-1}(\mathcal{C}-b)$ and $\tilde{\mathcal{C}}_i:=A^{-1}(\mathcal{C}_i-b)$ and the function $\tilde f:\tilde{\mathcal{C}}\to\mathbb{R}$ by
\[
\tilde f(z):=f(\Psi(z)),\qquad z\in\tilde{\mathcal{C}}.
\]
Then $\tilde f$ is piecewise–star–convex on $\tilde{\mathcal{C}}$ with respect to
the partition $\tilde{\mathcal{P}}=\{\tilde{\mathcal{C}}_i\}_{i\in I}$.
\end{proposition}

\begin{proof}
Since the matrix $A$ is invertible, the function $\Psi$ is a bijective affine homeomorphism with inverse
$\Psi^{-1}(x)=A^{-1}(x-b)$. If $z\in\tilde{\mathcal{C}}$ then $x:=\Psi(z)=Az+b\in\mathcal{C}$,
so $\tilde f(z)=f(x)$ is well defined. Because $A^{-1}$ and the translation
$x\mapsto x-b$ are continuous linear/affine maps, each
$\tilde{\mathcal{C}}_i=A^{-1}(\mathcal{C}_i-b)$ is   convex and $\bigcup_{i}\tilde{\mathcal{C}}_i=\tilde{\mathcal{C}}$. Thus
$\tilde{\mathcal{P}}$ is a partition of $\tilde{\mathcal{C}}$. In addition, for each $i\in I$, let $X_i^*:=\arg\min_{\mathcal{C}_i} f$ and
$\tilde X_i^*:=\arg\min_{\tilde{\mathcal{C}}_i}\tilde f$.
We claim
\begin{equation}\label{eq:min-map-AT}
\tilde X_i^*=A^{-1}(X_i^*-b).
\end{equation}
Indeed, for $z\in\tilde{\mathcal{C}}_i$ write $x=\Psi(z)=Az+b\in\mathcal{C}_i$. Then
$\tilde f(z)=f(x)$ and hence
\[
\min_{z\in\tilde{\mathcal{C}}_i}\tilde f(z)=\min_{x\in\mathcal{C}_i} f(x),
\qquad
z\in\tilde X_i^*\ \Longleftrightarrow\ x=\Psi(z)\in X_i^*.
\]
Equivalently, $z=A^{-1}(x-b)$ with $x\in X_i^*$, proving \eqref{eq:min-map-AT}.

Now, fix $i\in I$ and $z\in\tilde{\mathcal{C}}_i$, and set $x:=\Psi(z)=Az+b\in\mathcal{C}_i$. Because $f$ is piecewise-star-convex on $\mathcal{C}_i$, there exists $x_x^*\in X_i^*$ such that, for all $\lambda\in[0,1]$,
\begin{equation}\label{eq:star-x}
f\big((1-\lambda)x+\lambda x_x^*\big)
\leq  (1-\lambda)f(x)+\lambda f(x_x^*).
\end{equation}
Define $z_z^*:=A^{-1}(x_x^*-b)$. By \eqref{eq:min-map-AT}, $z_z^*\in\tilde X_i^*$. Using the affinity of $\Psi$, we conclude that 
\[
\Psi\big((1-\lambda)z+\lambda z_z^*\big)
=(1-\lambda)\Psi(z)+\lambda \Psi(z_z^*)
=(1-\lambda)x+\lambda x_x^*.
\]
Therefore,
\[
\begin{aligned}
\tilde f\big((1-\lambda)z+\lambda z_z^*\big)
&= f\!\left(\Psi\big((1-\lambda)z+\lambda z_z^*\big)\right)
= f\big((1-\lambda)x+\lambda x_x^*\big) \\
&\le (1-\lambda)f(x)+\lambda f(x_x^*)
= (1-\lambda)\tilde f(z)+\lambda \tilde f(z_z^*),
\end{aligned}
\]
where the inequality uses \eqref{eq:star-x} and the identities
$f(x)=\tilde f(z)$, $f(x_x^*)=\tilde f(z_z^*)$. Hence, for every
$z\in\tilde{\mathcal{C}}_i$ there exists $z_z^*\in\tilde X_i^*$ such that the
star–type inequality holds on $\tilde{\mathcal{C}}_i$. Since $i$ was arbitrary,
$\tilde f$ is piecewise–star–convex on $\tilde{\mathcal{C}}$ with respect to
$\tilde{\mathcal{P}}$.
\end{proof}

The next result establishes a first-order inequality that holds for any 
piecewise-star-convex. This property will play a key role in the convergence 
analysis of our algorithm.

\begin{proposition}\label{pr:pscf-pwsc}
Let $f:\mathbb{R}^n\to\mathbb{R}$ be locally Lipschitz on a convex set ${\cal C}$.
Assume $f$ is piecewise–star–convex on ${\cal C}$ with respect to the partition
${\cal P}=\{\C_i\}_{i\in I}$ of Definition~\ref{def:wsc-piecewise}. 
Fix $i\in I$ and $x\in\C_i$. Then there exists $x_x^*\in X_i^*:=\arg\min_{\C_i} f$ a minimizer on the cell $\C_i$ such that
\begin{equation}\label{eq:first-order-pwsc}
    f(x_x^*)-f(x)\ \ge\ f^{\circ}(x;\,x_x^*-x)\ \ge\ v^{\top}(x_x^*-x)
    \qquad \forall\, v\in \partial_c f(x),
\end{equation}
where $f^{\circ}$ denotes Clarke’s generalized directional derivative and
$\partial_c f(x)$ the Clarke subdifferential at $x$.
\end{proposition}

\begin{proof}
By piecewise–star–convexity on $\C_i$, there exists $x_x^*\in X_i^*$ such that
\[
  f\bigl(x+\lambda(x_x^*-x)\bigr)\ \le\ (1-\lambda)f(x)+\lambda f(x_x^*),\qquad \forall\,\lambda\in[0,1].
\]
Subtract $f(x)$, divide by $\lambda>0$, and let $\lambda$ goes to $0$ to obtain
\(
  f^{\circ}(x;\,x_x^*-x)\ \le\ f(x_x^*)-f(x).
\)
By the Clarke calculus, e.g., Theorem~\ref{th:cdd},  $f^{\circ}(x;d)=\max_{v\in\partial_c f(x)} v^{\top}d$, hence 
$v^{\top}(x_x^*-x)\le f^{\circ}(x;\,x_x^*-x)$ for all $v\in\partial_c f(x)$,
which proves \eqref{eq:first-order-pwsc}.
\end{proof}
The next result shows that, under piecewise-star-convexity, every Clarke stationary  point of the constrained problem is in fact a minimizer of $f$ on the cell to which it belongs.
\begin{corollary}\label{cor:stationary-cellwise}
Let $f:\mathbb{R}^{n}\to\mathbb{R}$ be locally Lipschitz and 
piecewise-star-convex on ${\cal C}$ with respect to the partition 
${\cal P}=\{\C_i\}_{i\in I}$. 
Suppose that $\bar x\in{\cal C}$ is a  Clarke stationary for problem~\eqref{pr:main}, i.e., there exists $v\in\partial_c f(\bar x)$ such that 
\begin{equation}\label{eq:stationarity-C}
  v^{\top}(x-\bar x) \geq 0  \qquad  \forall\,x\in{\cal C}.
\end{equation}
In addition, for any $i_{*}\in I$ such that $\bar x\in{\cal C}_{i_{*}}$, then $\bar x$ is a global minimizer of $f$ on the corresponding cell ${\cal C}_{i_{*}}$, namely,
\(
   \bar x \in X_{i_{*}}^* := \arg\min_{x\in{\cal C}_{i_{*}}} f(x).
\)
In particular, this means that
\(
   f(\bar x)=\min_{x\in{\cal C}_{i_{*}}} f(x).
\)
\end{corollary}
\begin{proof}
Apply Proposition~\ref{pr:pscf-pwsc} with $x=\bar x$, $v\in\partial_c f(\bar x)$  and $i=i_{*}$ to obtain some 
$x_{\bar x}^*\in X_{i_{*}}^*$ with
\[
   f(x_{\bar x}^*)-f(\bar x) \geq v^{\top}(x_{\bar x}^*-\bar x).
\]
By condition  \eqref{eq:stationarity-C}, taking $x=x_{\bar x}^*\in\C_{i_{*}}\subset\C$, the right-hand side is nonnegative, hence we conclude that  $f(x_{\bar x}^*)\ge f(\bar x)$. Since $x_{\bar x}^*\in X_{i_{*}}^*$ minimizes $f$ on $\C_{i_{*}}$, we also have $f(x_{\bar x}^*)\le f(\bar x)$. Therefore, we have $f(x_{\bar x}^*)=f(\bar x)$, which implies that   $\bar x\in X_{i_{*}}^*$.
\end{proof}

\begin{remark}
In the piecewise setting, $X_i^*=\arg\min_{\C_i} f$ consists of minimizers of $f$ \emph{restricted to the cell} $\C_i$. Such points are global minimizers on $\C_i$, but they need not be minimizers of $f$ on the whole set ${\cal C}$ and may fail to be even local minimizers in ${\cal C}$. This is precisely the situation in which piecewise star-convexity is useful,  star-type inequalities are enforced \emph{within each cell} with respect to $X_i^*$, even though they may break down across cell boundaries. Corollary~\ref{cor:stationary-cellwise} refines the global picture by showing that Clarke stationarity implies \emph{cellwise} optimality, namely $\bar x\in X_i^*$ for the cell $\C_i$ containing $\bar x$. In particular, if that cell achieves the smallest cell minimum value, that is, if $\min_{\C_i} f=\min_{\C} f$, then $\bar x$ is a global minimizer of $f$ on ${\cal C}$.
\end{remark}

\section{Frank--Wolfe algorithm} \label{Sec:FW}
In this section, we present the  Frank--Wolfe algorithm for solving~\eqref{pr:main} under   a Lipschitz-based adaptive stepsize rule. We also establish iteration-complexity guarantees under piecewise star-convexity, see \cite{Beck2015,BeckTeboulle2009}. Our analysis shows that, for piecewise star-convex functions  with a nonsmooth DC structure, the method attains an ${\cal O}(1/k)$ rate for the function values (relative to a best cell minimum) and a corresponding ${\cal O}(1/k)$ bound for the Frank--Wolfe gap. For comparison, a related Frank--Wolfe variant for~\eqref{pr:main} was proposed in \cite{Khamaru2019}, with stepsizes based on a curvature/Lipschitz-type constant. For general nonconvex objectives, the analysis in \cite{Lacoste2016} yields an ${\cal O}(1/\sqrt{k})$ rate in terms of stationarity.

Throughout this section, we assume access to a linear optimization oracle (LO oracle) that minimizes linear functions over ${\cal C}$.
The algorithm is stated below.\\

\hrule
\begin{algorithm} {\bf Frank--Wolfe$_{{C},f:=g-h}$ algorithm} \label{dAlg:CondGdf}
\begin{footnotesize}
\begin{description}
\item[Step 0.] Select $ x^0 \in {{\cal C}}$ and  $L_0>0$. Set $k=0$.
\item [Step 1.]  Take $u^k\in \partial h(x^k)$.  Set $j:=\min\{\ell \in {\mathbb N}:~2^{\ell}L_k \geq  2L_0\}$.
\item [Step 2.]    Use an ``LO oracle" to compute an optimal solution $p^{k}$ and the optimal value ${\omega}_k$ as  follows 
\begin{equation} \label{deq:CondG}
p^{k} \in {\argmin}_{p\in {{\cal C}}} (\nabla g(x^k)-u^k)^{\T} (p-x^k),\qquad  \quad {\omega}_k:= (\nabla g(x^k)-u^k)^{\T}(p^{k}-x^k).
\end{equation}
\item[Step 3.]  If ${\omega}_k=0$, then {\bf stop}. Otherwise, compute the step size $\lambda_{j} \in (0, 1]$ as follows 
 \begin{equation}\label{deq:fixed.step}
\lambda_{j}=\mbox{min}\left\{1, \frac{|{\omega}_k|}{{2^{j}L_k}  \|p^{k}-x^k\|^2}\right\}
                 :={\argmin}_{\lambda \in (0,1]}\left \{-|{\omega}_k| \lambda+\frac{{2^{j}L_k} }{2} \|p^{k} -x^k\|^2 \lambda^2 \right \}.
\end{equation}
\item[Step 4.] If
\begin{equation}\label{deq:test}
 f(x^k+ \lambda_{j}(p^{k}-x^k)) \leq f(x^k) -|{\omega}_k| \lambda_{j}+\frac{{2^{j}L_k} }{2} \|p^{k} -x^k\|^2 \lambda_j^2, 
\end{equation}
then set $j_k=j$ and go to {\bf Step  5}. Otherwise, set $j=j+1$ and go to {\bf Step  3}.

\item[Step 5.] Set  $\lambda_k:= \lambda_{j_k}$ and  define the next iterate $x^{k+1}$ and the next approximation to the Lipschitz constant   $L_{k+1}$ as follows  
\begin{equation}\label{deq:iteration}
x^{k+1}:=x^k+ \lambda_{k}(p^{k}-x^k),  \qquad \quad L_{k+1}:=2^{j_k-1}L_k.
\end{equation}
Set $k\gets k+1$, and go to {\bf Step  1}.
\end{description}
\hrule
\end{footnotesize}
\end{algorithm}
\vspace{0.3cm}

The LO oracle provides $p^k\in{\cal C}$ and induces the direction $d^k:=p^k-x^k$, with $p^k:=p(x^k, u^k)$ as defined in in~\eqref{eq:LO}. The scalar quantity $\omega_k:=\omega(x^k, u^k)\leq 0$, where  $\omega(x^k, u^k)$ as defined in~\eqref{eq:LO} is the corresponding Frank--Wolfe gap and vanishes only at stationary points, see Proposition~\ref{prop:omega}. Therefore, the algorithm stops whenever $\omega_k=0$. From now on, we assume 
\begin{equation} \label{deq:vjn}
\omega_k<0,  \qquad  \qquad k=0, 1, \ldots.
\end{equation}
 for all $k\ge 0$, so the {\it method generates an infinite sequence $\{x^{k}\}_{k\in\mathbb{N}}\subset{\cal C}$}. In fact,  since the update of $x_{k+1}$ in ~\eqref{deq:iteration} is a convex combination of $x^{k}$ and $p^{k}$ with $\lambda_k\in(0,1]$, the convexity of ${\cal C}$ yields $x^{k}\in{\cal C}$ for every $k$. 

Next we apply  Proposition~\ref{le:DescentLemma} to show that the backtracking procedure in Step  2 and 3 terminates in finitely many trials at each iteration.
Therefore, $j_k$, $\lambda_k$, and the update $L_{k+1}$ in~\eqref{deq:iteration} are well defined.
The next proposition states this formally and records the key descent inequality used in the complexity analysis.

\begin{proposition} \label{dpr:mainDf}
For all $j\in {\mathbb N}$  such that   ${2^{j}L_k}\geq L$, the inequality \eqref{deq:test} holds. Consequently, the number  $j_k$ in {\bf Step  4}  is well defined. Furthermore,   $j_k$ is  the smallest non-negative integer satisfying   the following two conditions
  \begin{equation} \label{eq:TestLineSeachC1}
 2^{j_k}L_k \geq  2L_0, 
 \end{equation}
 \begin{equation}\label{eq:TestLineSeachC2}
 f(x^k+ \lambda_{k}(p^{k}-x^k)) \leq f(x^k) -| {\omega}_{k}|\lambda_{k}+\frac{{2^{j_k}L_k} }{2} \|p^{k} -x^k\|^2\lambda_{k}^2. 
\end{equation}
 In addition, the sequence $\{x^k\}_{k\in\mathbb{N}}$ generated by Algorithm~\ref{dAlg:CondGdf} is well defined.  And   the following inequality holds
 \begin{equation} \label{deq:testsDf}
 f(x^{k+1}) \leq f(x^k)- \frac{1}{2} |{\omega}_{k}| \lambda_k, \qquad  \qquad k=0, 1, \ldots.
\end{equation}
 \end{proposition}
  \begin{proof}
 First  we will prove that  \eqref{deq:test} holds   for  all  $j\in {\mathbb N}$ such that ${2^{j}L_k}\geq L$.  It follows from  Proposition~\ref{le:DescentLemma}  that  
 \begin{equation} \label{deq:auxfrg}
   g(x^k+ \lambda_{j}(p^{k}-x^k)) \leq g(x^k) + \nabla g(x^k)^{\T}(p^{k}-x^k) \lambda_{j} + \frac{L}{2}\|p^{k}-x^k\|^2\lambda_{j}^{2} .
\end{equation}
 On the other hand,  considering that $u^k\in \partial h(x^k)$,  it follows from Proposition~\eqref{le:cc}  that 
 \begin{equation} \label{deq:auxfrh}
 h(x^k+ \lambda_{j}(p^{k}-x^k))\geq h(x^k) + \lambda_{j} \langle u^k, p^{k}-x^k\rangle.
\end{equation}
Thus, combining \eqref{deq:auxfrg} with \eqref{deq:auxfrh} and taking into account  the equality in \eqref{deq:CondG} and \eqref{deq:vjn}, some algebraic manipulations shows that 
  \begin{equation*}  
  f(x^k+ \lambda_{j}(p^{k}-x^k)) \leq f(x^k) -|{\omega_k}| \lambda_{j} + \frac{{2^{j}L_k}}{2}\|p^{k}-x^k\|^2\lambda_{j}^{2} + \frac{L-{2^{j}L_k}}{2}\|p^{k}-x^k\|^2\lambda_{j}^{2}.
 \end{equation*}
 Hence,  it follows from the last inequality that   \eqref{deq:test} holds,  for all $j\in {\mathbb N}$  such that   ${2^{j}L_k}\geq L$, which proves the first statement. Consequently,  $j_k$   is well defined  and  is  the smallest non-negative integer satisfying    \eqref{eq:TestLineSeachC1} and \eqref{eq:TestLineSeachC2}.  Therefore,  the sequence $\{x^k\}_{k\in\mathbb{N}}$ generated by Algorithm~\ref{dAlg:CondGdf} is well defined. We  proceed  with the proof of \eqref{deq:testsDf}, by showing  the following inequality 
   \begin{equation} \label{deq:auxfr}
-|{\omega_k}| \lambda_{k} + \frac{{2^{j_k}L_k}}{2}\|p^{k}-x^k\|^2\lambda_{k}^{2}\leq   - \frac{1}{2} |{\omega_k}| \lambda_k.
\end{equation}
  For that, we note that $0\leq \lambda_{k}\geq |{\omega_k}|/{( 2^{j_k}L_k \|p^{k}-x^k\|^2)}$, from which we obtain that $ 2^{j_k}L_k \|p^{k}-x^k\|^2\leq |{\omega_k}|/{\lambda_{k}}$. Substituting this into the left hand side of \eqref{deq:auxfr} yields the inequality.
Hence, \eqref{deq:auxfr} holds. Therefore, \eqref{deq:testsDf}  follows by combining \eqref{eq:TestLineSeachC2}  with \eqref{deq:auxfr}, 
 which  concludes the proof. 
 \end{proof} 
 
\subsection{Convergence analysis} \label{Sec:StarConvexConvProp}
In this section we analyze Algorithm~\ref{dAlg:CondGdf} under Assumption~{\bf(A)}. We first show that the backtracking rule yields uniformly bounded Lipschitz estimates and a uniform lower bound on the accepted stepsizes. These bounds, combined with the sufficient decrease inequality in Proposition~\ref{dpr:mainDf}, imply monotone descent of $\{f(x^k)\}_{k\in\N}$, $\lim_{k\to\infty}\omega_k=0$, and Clarke-stationarity of every accumulation point. If $f$ is piecewise-star-convex, any such limit point is a global minimizer on its cell and a local minimizer of $f$  whenever it lies in the interior of that cell.

 For simplifying the notations we define the following constants 
\begin{equation}\label{deq:alpha}
\alpha:=2\,(L+L_0)\,\diam({\cal C})^{2}>0.
\end{equation}
The next lemma provides a uniform upper bound on the backtracking estimates produced,
and yields a convenient lower bound on the accepted stepsizes $\lambda_k$.

\begin{lemma}\label{dle:bfbwg}
Assume {\bf(A)} and let $(L_k)_{k\in\mathbb{N}}$ be generated by~\eqref{deq:iteration}.
Then
\begin{equation}\label{deq:Lk}
0<L_k\le L+L_0,\qquad k=0,1,\ldots .
\end{equation}
Moreover, the stepsizes $(\lambda_k)_{k\in\mathbb{N}}$ selected by Lipschitz-based adaptive stepsize  satisfy
\begin{equation}\label{deq:blk}
\lambda_k\ \ge\ \min\Bigl\{1,\ \frac{|{\omega}^{k}|}{\alpha}\Bigr\},\qquad k=0,1,\ldots .
\end{equation}
\end{lemma}

\begin{proof}
We first prove~\eqref{deq:Lk}. Since $L_0>0$, we have $0<L_0\le L+L_0$.
Assume that $0<L_k\le L+L_0$ for some $k\ge 0$ and recall from~\eqref{deq:iteration} that
$
L_{k+1}=2^{\,j_k-1}L_k .
$
If $j_k=0$, then $L_{k+1}=\tfrac12 L_k$, hence $0<L_{k+1}\le L_k\le L+L_0$. If $j_k\ge 1$, suppose by contradiction that $L_{k+1}>L+L_0$.
Then $2^{\,j_k-1}L_k=L_{k+1}>L$,  which implies that $2^{\,j_k-1}L_k\ge L$. By the first statement of Proposition~\ref{dpr:mainDf}, the backtracking inequality~\eqref{deq:test} holds for $j=j_k-1$. This contradicts the fact that $j_k$ is the first index accepted by the backtracking loop in Step~3.2 of Lipschitz-based adaptive stepsize procedure. Hence $L_{k+1}\le L+L_0$, and in all cases $0<L_{k+1}\le L+L_0$. By induction,~\eqref{deq:Lk} holds for all $k$. We now prove~\eqref{deq:blk}. By~\eqref{deq:fixed.step} with $j=j_k$,
\[
\lambda_k=\min\Bigl\{1,\ \frac{|\omega_k|}{2^{\,j_k}L_k\,\|p^k-x^k\|^{2}}\Bigr\}.
\]
Since $x^k,p^k\in{\cal C}$, we have $\|p^k-x^k\|\le \diam({\cal C})$ by~\eqref{eq:diam}. Moreover, from~\eqref{deq:iteration} we obtain $2^{\,j_k}L_k=2L_{k+1}$, and then~\eqref{deq:Lk} yields
$
2^{\,j_k}L_k=2L_{k+1}\le 2(L+L_0).
$
Combining these estimates gives
\[
\lambda_k
\ge
\min\Bigl\{1,\ \frac{|\omega_k|}{2(L+L_0)\,\diam({\cal C})^{2}}\Bigr\}
=
\min\Bigl\{1,\ \frac{|\omega_k|}{\alpha}\Bigr\},
\]
where $\alpha$ is defined in~\eqref{deq:alpha}. This is~\eqref{deq:blk} and the proof is concluded. 
\end{proof}

In the next theorem, we summarize the resulting convergence guarantees for Algorithm~\ref{dAlg:CondGdf}.

\begin{theorem}\label{dcr:pgfic}
Let ${\cal C}\subset\R^n$ be compact and convex, and let $\{x^k\}_{k\in\N}$ be generated by
Algorithm~\ref{dAlg:CondGdf}. Then, the following statement holds: 
\begin{itemize}
\item[(i)] The objective values sequence $\{f(x^k)\}_{k\in\N}$ is nonincreasing. Moreover, the  Frank--Wolfe gap satisfies $\lim_{k\to\infty}\omega_k=0.$
\item[(ii)] Every accumulation point $\bar x$ of  $\{x^k\}_{k\in\N}$ is Clarke-stationary for~\eqref{pr:main}; 
\item[(iii)] If, in addition, $f$ is piecewise-star-convex on ${\cal C}$ with respect to a partition ${\cal P}=\{{\cal C}_i\}_{i\in I}$,
then for any such limit point $\bar x$ and any $i_*\in I$ with $\bar x\in{\cal C}_{i_*}$ and $\bar x\in X_{i_*}^*$,
i.e., $\bar x$ is a global minimizer of $f$ on its cell ${\cal C}_{i_*}$. Furthermore, if  $\bar x \in {\rm int}\, {\cal C}_{i_*}$, where ${\rm int}\, {\cal C}_{i_*}$ denotes the interior of the set ${\cal C}_{i_*}$, then  $\bar x$ is local minimizer of $f$.
\end{itemize} 
\end{theorem}

\begin{proof}
To prove item~(i), first recall  that $\omega_k\neq 0$ for all $k$,  which implies  that an infinite sequence is generated.  By Proposition~\ref{dpr:mainDf}, the accepted stepsize satisfies the sufficient decrease inequality~\eqref{deq:testsDf}, and hence
$
0\leq \frac{1}{2} |{\omega}_{k}|\leq f(x^k)- f(x^{k+1})
$
for all $k=0,1,\ldots,$  which implies that $\{f(x^k)\}$ is monotone nonincreasing. We proceed to prove that $\lim_{k\to\infty}\omega_k=0$.
Let $\alpha>0$ be defined in~\eqref{deq:alpha}. Combining the sufficient decrease
\eqref{deq:testsDf} with the lower bound~\eqref{deq:blk}, we obtain
\begin{equation}\label{deq:ibbbc}
 \tfrac{1}{2} \min\bigl\{|{\omega}_{k}|,  \tfrac{1}{\alpha}|{\omega}_{k}|^2\bigr\} \leq f(x^k) -  f(x^{k+1}),  \qquad k=0,1,  \ldots.
\end{equation}
Since $\{f(x^k)\}_{k\in\N}$ is monotone nonincreasing and bounded below on ${\cal C}$, it converges.  Hence, we have 
$\lim_{k\to\infty} (f(x^k)-f(x^{k+1}))= 0$. Therefore,  passing to the limit in inequality ~\eqref{deq:ibbbc} yields
$
\lim_{k\to +\infty} \min\bigl\{|{\omega}_{k}|,  \tfrac{1}{\alpha}|{\omega}_{k}|^2\bigr\} =0, 
$
which implies  that $\lim_{k\to\infty}\omega_k=0$.

To prove item~(ii), let $\bar x$ be a limit point of $\{x^k\}_{k\in\N}$ and take a subsequence $\{x^{k_\ell}\}_{\ell\in\N}$ such that
$\lim_{\ell \to +\infty}x^{k_\ell}= \bar x$. Because $h$ is convex and ${\cal C}$ is compact, Proposition~\ref{pr:bsubdiff} ensures that
$\bigcup_{x\in{\cal C}}\partial h(x)$ is bounded. Hence,  $\{u^{k_\ell}\}_{\ell\in\N}$ is bounded and, passing to a further subsequence if needed,
we may assume $\lim_{\ell \to +\infty} u^{k_\ell}= \bar u$,  for some $\bar u\in\R^n$. By Proposition~\ref{pr:csubdiff}, the graph of subdifferential  $\partial h$ is closed, which implies that  $\bar u\in\partial h(\bar x)$. Now fix any $x\in{\cal C}$. Since $p^{k_\ell}$ minimizes the linear functional in~\eqref{deq:CondG}, we have
\[
(\nabla g(x^{k_\ell})-u^{k_\ell})^{\top}(x-x^{k_\ell})
\;\ge\;
(\nabla g(x^{k_\ell})-u^{k_\ell})^{\top}(p^{k_\ell}-x^{k_\ell})
=\omega_{k_\ell}.
\]
Using that  $\lim_{\ell \to +\infty} x^{k_\ell}= \bar x$, $\lim_{\ell \to +\infty}u^{k_\ell}=\bar u$, the continuity of $\nabla g$,
and $\lim_{\ell \to +\infty}\omega_{k_\ell}= 0$, we obtain that  $(\nabla g(\bar x)-\bar u)^{\top}(x-\bar x)\ge 0$,  for all $x\in{\cal C}$, 
which is precisely the Clarke-stationarity condition for~\eqref{pr:main}.

Finally, we prove item~(iii). Assume additionally that $f$ is piecewise-star-convex on ${\cal C}$ with respect to ${\cal P}=\{{\cal C}_i\}_{i\in I}$. Let $\bar x$ be any accumulation point of $\{x^k\}_{k\in\N}$. Choose an index $i_*\in I$ such that $\bar x\in{\cal C}_{i_*}$, such an index exists because ${\cal P}$ covers ${\cal C}$. Since item~(ii) has been proved, $\bar x$ is Clarke-stationary for~\eqref{pr:main}. Therefore, Corollary~\ref{cor:stationary-cellwise} applies and yields $\bar x\in X_{i_*}^*$, i.e., $\bar x$ is a global minimizer of $f$ over ${\cal C}_{i_*}$. It remains to show the local minimality statement when $\bar x\in{\rm int}\,{\cal C}_{i_*}$. Because ${\rm int}\,{\cal C}_{i_*}$ is open in $\R^n$, there exists $r>0$ such that $B(\bar x,r)\subset {\cal C}_{i_*}$. For every $y\in{\cal C}\cap B(\bar x,r)$ we have $y\in{\cal C}_{i_*}$, and since $\bar x$ minimizes $f$ on ${\cal C}_{i_*}$ it follows that $f(\bar x)\le f(y)$, for all $ y\in{\cal C}\cap B(\bar x,r)$.  Hence $\bar x$ is a local minimizer of $f$ on ${\cal C}$.
\end{proof}

We end this section by emphasizing that, compared with the schemes in \cite{BeckTeboulle2004, Khamaru2019}, the only algorithmic ingredient that differs is the stepsize selection performed by the backtracking procedure in Algorithm~\ref{dAlg:CondGdf}. In the next section, we use the resulting descent estimate to derive the corresponding convergence and iteration--complexity guarantees.

\subsection{Iteration-complexity analysis} \label{Sec:StarConvexIntComplexity}
In this section, we derive iteration--complexity bounds for the sequence $\{x^k\}_{k\in\mathbb{N}}$ generated by Algorithm~\ref{dAlg:CondGdf} when applied to the DC problem~\eqref{pr:main}. We focus on objectives of the form $f=g-h$ that are \emph{piecewise star-convex} on ${\cal C}$, where $g$ is convex and continuously differentiable with $L$-Lipschitz continuous gradient on ${\cal C}$, and $h$ is convex and possibly nonsmooth.

We begin with a structural lemma for finite convex partitions that requires only continuity of $f$ and monotonicity of the objective values.
It states that, if the function values converge to the minimum attained on some cell of the partition, then after a finite number of iterations the sequence can only visit cells whose minimum value does not exceed that limit.
This prefix--exclusion mechanism will be used later, once the descent analysis ensures that $\{f(x^k)\}_{k\in \mathbb N}$ is nonincreasing.

\begin{lemma}\label{lem:prefix-exclusion}
Let ${\cal C}\subset\mathbb{R}^n$ be compact and convex, and let
${\cal P}=\{ {\cal C}_i : i=1,\dots,m\}$ be a finite partition of ${\cal C}$ into nonempty convex sets.
Let $f:\mathbb{R}^{n}\to\mathbb{R}$ be such that $f_i^*:=\min_{x\in{\cal C}_i} f(x)$ is well defined for every $i$.
Let $\{x^k\}_{k\in\mathbb{N}}\subset{\cal C}$ be  a  sequence  satisfying $f(x^{k+1})\le f(x^k)$ for all $k$, and suppose that, for some
$i_*\in\{1,\dots,m\}$,
\begin{equation}\label{eq:limit-f-seq-lemma}
  \lim_{k\to\infty} f(x^k) \;=\; f_{i_*}^* .
\end{equation}
Then there exists $N\in\mathbb{N}$ such that, for all $k\ge N$ and every index $i_k\in \{1,\dots,m\}$ with $x^k\in{\cal C}_{i_k}$,
\[
  f_{i_k}^* \;\le\; f_{i_*}^* .
\]
\end{lemma}

\begin{proof}
Define
$
\hat I:=\{\, i\in\{1,\dots,m\}:\ f_i^*>f_{i_*}^* \,\}.
$
If $\hat I=\emptyset$, then  the conclusion holds  by taking  $N:=0$. Assume $\hat I\neq\emptyset$ and fix $i\in\hat I$.
By \eqref{eq:limit-f-seq-lemma}, since $f_i^*>f_{i_*}^*$, there exists $N_i\in\mathbb{N}$ such that
$
f(x^{N_i})<f_i^*.
$
Because the sequence $\{f(x^k)\}_{k\in\mathbb{N}}$ is nonincreasing, it follows that
\[
f(x^k)\le f(x^{N_i})<f_i^*,\qquad \forall\,k\ge N_i.
\]
If, for some $k\ge N_i$, we had $x^k\in{\cal C}_i$, then by definition of $f_i^*$ as the minimum of $f$ on ${\cal C}_i$
we would have $f(x^k)\ge f_i^*$, contradicting the strict inequality above. Hence,
\[
x^k\notin{\cal C}_i,\qquad \forall\,k\ge N_i.
\]
Now set $N:=\max_{i\in\hat I} N_i$, finite because $\hat I$ is finite. Then for every $k\ge N$ we have
$x^k\notin{\cal C}_i$ for all $i\in\hat I$. Therefore, if $x^k\in{\cal C}_{i_k}$ with $k\ge N$, necessarily
$i_k\notin\hat I$, which means $f_{i_k}^*\le f_{i_*}^*$, and the desired inequality is proved. 
\end{proof}

We now state a complexity result showing that, despite nonconvexity, piecewise–star–convex objectives achieve the same Frank–Wolfe iteration bound as in the convex setting.

\begin{theorem} \label{th:fcr}
Let ${\cal C}\subset\mathbb{R}^n$ be compact and convex, and let ${\cal P}=\{ {\cal C}_i : i=1,\dots,m\}$ be a finite partition of ${\cal C}$ into nonempty convex sets. Assume that $f:\mathbb{R}^{n}\to\mathbb{R}$ is piecewise-star-convex on ${\cal C}$. For each $i\in\{1,\dots,m\}$, let $X_i^*$ denote the set of global minimizers of $f$ on ${\cal C}_i$, and let
\[
   f_i^* := \min_{x\in {\cal C}_i} f(x)
\]
be the corresponding minimum value. Let $\{x^k\}_{k\in\mathbb{N}}$ be the sequence generated by Algorithm~\ref{dAlg:CondGdf}. Then there exist $i_{*} \in \{1,\dots,m\}$ and $N\in {\mathbb N}$ such that:
\begin{itemize}
\item[(i)] 	$f(x^k)-f_{i_{*}}^*  \leq \displaystyle   \frac{4(L+L_0)\diam({\cal C})^2}{k-N+1}$, for all  $k\geq N$;
\item[(ii)] $\displaystyle \min_{\ell \in \left\{ \lfloor {k-N+1}/{2}\rfloor+N+1,...,k \right\} } |{\omega}_{\ell}| \leq \frac{16(L+L_0)\diam({\cal C})^2}{k-N-1}$, for all $ k\geq N+2$. 
\end{itemize}
\end{theorem}

\begin{proof}
Using \eqref{eq:TestLineSeachC2} in Proposition~\ref{dpr:mainDf} and recalling that $\omega_k\le 0$, we obtain that 
\begin{equation}\label{eq:ascfa}
  f\!\big(x^k+\lambda_k(p^{k}-x^k)\big)
  \le
  f(x^k)-|\omega_k|\lambda_k+\frac{2^{\,j_k}L_k}{2}\|p^{k}-x^k\|^2\,\lambda_k^{2}.
\end{equation}
On the other hand, by definition of the Lipschitz-based adaptive stepsize $ \lambda_k$, we have 
\[
  \lambda_k =\operatorname*{arg\,min}_{\lambda\in(0,1]}
  \Big\{-|\omega_k|\,\lambda + \tfrac{2^{\,j_k}L_k}{2}\,\|p^{k}-x^k\|^2\,\lambda^2 \Big\}.
\]
Hence, for any $\beta_k\in(0,1]$, in particular  for $\beta_k:=\tfrac{2}{k+2}$, we conclude from the last equality that 
\[
-|\omega_k|\,\lambda_k+ \frac{2^{\,j_k}L_k}{2}\|p^{k}-x^k\|^2\,\lambda_k^{2}
\le
-|\omega_k|\,\beta_k+ \tfrac{2^{\,j_k}L_k}{2}\,\|p^{k}-x^k\|^2\,\beta_k^2.
\]
Thus, combining this previous inequality with the inequality in  \eqref{eq:ascfa}, we conclude  that 
\[
  f(x^{k+1})
  \le
  f(x^k)-|\omega_k|\,\beta_k
  +\frac{2^{\,j_k}L_k}{2}\,\|p^{k}-x^k\|^2\,\beta_k^{2}.
\]
Since $\|p^{k}-x^k\|\le \diam({\cal C})$ and $\frac{2^{\,j_k}L_k}{2}=L_{k+1}\le L+L_0$, by \eqref{deq:iteration} and Lemma~\ref{dle:bfbwg}, we obtain
\begin{equation}\label{eq:basic-rec}
  f(x^{k+1}) \le f(x^k) - |\omega_k|\,\beta_k + (L+L_0)\,\diam({\cal C})^2\,\beta_k^{2}.
\end{equation}
Set $A:=(L+L_0)\,\diam({\cal C})^2$. Since ${\cal C}$ is compact, $\{x^k\}_{k\in\mathbb{N}}$ is bounded and has accumulation points.
By Theorem~\ref{dcr:pgfic}\emph{(i)}, $\{f(x^k)\}_{k\in\mathbb{N}}$ is monotone nonincreasing and thus convergent. Let $\bar x$ be a limit point, and let
$(x^{k_\ell})_{\ell\in\mathbb{N}}$ be a subsequence such that \(\lim_{\ell \to +\infty}x^{k_\ell}={\bar x}\).
By Theorem~\ref{dcr:pgfic}\emph{(iii)}, there exists an index $i_*\in\{1,\dots,m\}$ such that $\bar x\in{\cal C}_{i_*}$ and $\bar x\in X_{i_*}^*$. Hence, we have 
\[
\lim_{k\to\infty} f(x^k)=\lim_{\ell\to\infty} f(x^{k_\ell})=f(\bar x)=f_{i_*}^*,
\qquad\text{and}\qquad
f_{i_*}^*\le f(x^k)\ \ \forall\,k\in\mathbb{N}.
\]
Now applying  Lemma~\ref{lem:prefix-exclusion} to the monotone sequence $\{f(x^k)\}$ and the finite partition ${\cal P}$.we conclude that  there exists $N\in\mathbb{N}$ such that, for every $k\ge N$ and every index $i_k$ with $x^k\in{\cal C}_{i_k}$,
\begin{equation}\label{eq:cell-min-comp}
f_{i_k}^*\le f_{i_*}^*.
\end{equation}
Fix $k$ and choose $i_k$ with $x^k\in{\cal C}_{i_k}$. Select $x_k^*\in X_{i_k}^*$ such that $f(x_k^*)=f_{i_k}^*$.
Since $f$ is piecewise-star-convex on ${\cal C}$, Proposition~\ref{pr:pscf-pwsc} and Theorem~\ref{subdif_DC} yield
\begin{equation}\label{eq:apppsc}
 f_{i_k}^*-f(x^k) \ge (\nabla g(x^k)-u^k)^{\T}(x^*_{k} - x^k),
\end{equation}
for some $u^k\in\partial h(x^k)$. Combining \eqref{eq:cell-min-comp} with \eqref{eq:apppsc}, we obtain, for all $k\ge N$, that 

\begin{equation}\label{eq:apppscc}
 f_{i_*}^*-f(x^k) \ge (\nabla g(x^k)-u^k)^{\T}(x^*_{k} - x^k).
\end{equation}
Moreover, since $p^k$ solves the LO oracle \eqref{deq:CondG}, we conclude that for every $x\in{\cal C}$, there holds 
\[
(\nabla g(x^k)-u^k)^{\T}(x-x^k)\ \ge\ (\nabla g(x^k)-u^k)^{\T}(p^k-x^k)=\omega_k.
\]
Applying this with $x=x_k^*$ and using \eqref{eq:apppscc}, we obtain $0 \ge f_{i_*}^*-f(x^k) \ge \omega_k$ for all $k\ge N$. Hence, 
\begin{equation}\label{eq:a-le-b}
  0 \le f(x^k)-f_{i_*}^* \le |\omega_k| \qquad \forall\,k\ge N.
\end{equation}
For each integer ${m}\geq 1$, define
\begin{equation} \label{eq:auxabb}
  a_{m} := f\!\big(x^{N+{m}-1}\big) - f_{i_*}^*,\qquad
  b_{m} := \big|\omega_{\,N+{m}-1}\big|,\qquad
  \beta_{m} := \frac{2}{{m}+1}.
\end{equation} 
Applying \eqref{eq:basic-rec} with $A:=(L+L_0)\,\diam({\cal C})^2$,  $k=N+{m}-1$ and $\beta_k=\beta_{m}$, we get for all ${m}\ge 1$,
\[
  a_{{m}+1}\le a_{m} - b_{m}\,\beta_{m} + A\,\beta_{m}^2.
\]
Together with \eqref{eq:a-le-b}, we have $0\le a_{m}\le b_{m}$ for all ${m}\ge 1$.
Therefore the hypotheses of Lemma~\ref{lemma taxa} apply, yielding
\[
  a_{m} \le \frac{4A}{{m}}, \qquad \forall m\geq 1, \qquad \qquad \min_{\tau\in\{\lfloor {m}/2\rfloor+2,\ldots,{m}\}} b_\tau \le \frac{16A}{{m}-2} \qquad \forall m\geq 3.
\]
Returning to $k=N+{m}-1$ and considering  \eqref{eq:auxabb}, we conclude that for every $k\ge N$, there holds 
\[
  f(x^k)-f_{i_*}^* = a_{k-N+1}\le \frac{4A}{k-N+1},
\]
and for every $k\ge N+2$, there holds 
\[
  \min_{\ell\in\{\lfloor (k-N+1)/2\rfloor+N+1,\ldots,k\}} |\omega_\ell|
  = \min_{\tau\in\{\lfloor (k-N+1)/2\rfloor+2,\ldots,k-N+1\}} b_\tau
  \le \frac{16A}{k-N-1}.
\]
Since $A=(L+L_0)\diam({\cal C})^2$, these are exactly the bounds in items \emph{(i)}--\emph{(ii)}.
\end{proof}

The next remark rewrites the estimates of Theorem~\ref{th:fcr} in an explicit $\varepsilon$--complexity form.
In particular, it provides concrete iteration thresholds ensuring $\varepsilon$--accuracy for the objective residual, with respect to the limiting cell minimum, and for the Frank--Wolfe optimality certificate.

\begin{remark}\label{rk:eps-complexity}
Under the assumptions of Theorem~\ref{th:fcr}, there exist $i_*\in\{1,\ldots,m\}$ and $N\in\mathbb{N}$ such that the
following $\varepsilon$-complexity estimates hold. Given $\varepsilon>0$, any index $k$ satisfying
\[
k\ \ge\ N-1+\frac{4}{\varepsilon}(L+L_0)\diam({\cal C})^2
\]
guarantees
$
f(x^k)-f_{i_*}^*\ \le\ \varepsilon.
$
On the other hand, given $\varepsilon>0$, any index $k$ satisfying the following condition 
\[
k\ \ge\ N+1+\frac{16}{\varepsilon}(L+L_0)\diam({\cal C})^2
\]
guarantees the existence of an iterate index
\[
\ell\in\Bigl\{\bigl\lfloor \tfrac{k-N+1}{2}\bigr\rfloor+N+1,\ldots,k\Bigr\}
\qquad\text{such that}\qquad
|\omega_\ell|\ \le\ \varepsilon.
\]
Equivalently, after $k$ iterations one can extract, from the last $\mathcal{O}(k)$ iterates, an $\varepsilon$-certificate
for the Frank--Wolfe gap. In practice, we do not attempt to estimate $N$. Instead, we use the computable Frank--Wolfe gap as the stopping
criterion and terminate as soon as $|\omega_k|\le \varepsilon$ or, more robustly, when
$\min_{0\le t\le k}|\omega_t|\le \varepsilon$. Theorem~\ref{th:fcr} then guarantees that such a threshold will be reached
after a finite number of iterations and provides an $\mathcal{O}(1/\varepsilon)$ worst--case bound, up to the transient
index shift encoded by $N$.
\end{remark}

To orient the reader, we record a simple situation in which the complexity bounds of Theorem~\ref{th:fcr} apply from the first iterate, that is, one may take $N=0$.

\begin{remark}
If all cellwise minimum values coincide, i.e., $f_i^*=f_j^*$ for every $i,j$, then in Theorem~\ref{th:fcr} one may take $N=0$.
Indeed, no cell has a strictly larger minimum value, hence the index set $\hat I=\{\, i:\ f_i^*>f_{i_*}^* \,\}$ introduced in Lemma~\ref{lem:prefix-exclusion} is empty. Therefore, the exclusion step is not needed, and the comparison argument leading to \eqref{eq:apppscc} holds for every $k$ without any initial index shift. Consequently, the complexity bounds in items (i)--(ii) apply from the first iterate. In particular, this situation covers the globally star--convex case on ${\cal C}$, a single--cell partition, for which the unique cell is ${\cal C}$ itself and $f_1^*=f^*_{\cal C}$. Hence, Theorem~\ref{th:fcr} reduces to the standard convex--type $\mathcal{O}(1/k)$ guarantees for the objective residual $f(x^k)-f^*_{\cal C}$ and for the gap certificate $|\omega_k|$.
\end{remark}

Theorem~\ref{th:fcr} shows that, under piecewise--star--convexity, the Frank--Wolfe method equipped with our backtracking stepsize achieves the convex-case rate, namely, $f(x^k)-f^\star=\mathcal{O}(1/k)$, and the optimality certificate $|\omega_k|$ decays at the same order.
Thus, the proposed projection-free scheme enjoys the standard sublinear complexity guarantees while remaining compatible with the piecewise--star--convex framework.

\section{Conclusions} \label{sec:conclusions}
We analyzed Algorithm~\ref{dAlg:CondGdf} within the DC piecewise--star--convex framework.
In particular, for functions  that are piecewise--star--convex on a compact convex set, we showed that the proposed Frank--Wolfe scheme equipped with the backtracking stepsize attains the same sublinear iteration complexity as in the convex case, namely,  the function-value error (measured against the appropriate cellwise minimum) and the Frank--Wolfe gap both decay on the order of $1/k$. These results extend convex-type complexity guarantees to a broad class of nonconvex and potentially nonsmooth DC objectives, complementing recent rates established for globally star--convex functions.
Natural directions for further work include, identifying verifiable conditions under which the iterates enter (and remain in) a single cell so that the bounds hold from the first iteration,  refining the analysis to better quantify the transient phase before the ``active'' cell is identified, and extending the guarantees to common Frank--Wolfe enhancements such as away--steps and pairwise steps, as well as other line-search mechanisms.

\section*{Funding}
The first and third authors were supported by the Australian Research Council (ARC), Solving hard Chebyshev approximation problems through nonsmooth analysis (Discovery Project DP180100602). The second author was supported in part by CNPq - Brazil Grants 304666/2021-1 and Fapesc - TR nº 2024TR002238. This work was partially conducted while the second author visited the first and third authors at Deakin University in November 2022, with thanks to the host institution for funding and providing a stimulating scientific environment.

\bibliographystyle{habbrv}
\bibliography{FrankWolfeMethod}

\end{document}